\documentclass[11pt,reqno]{amsart}
\usepackage{amssymb,latexsym}
\usepackage{graphicx}

\theoremstyle{plain}
\numberwithin{equation}{section}
\newtheorem{thm}{Theorem}[section]
\newtheorem{theorem}[thm]{Theorem}

\begin{document}

\setcounter{page}{1}

\title[ An Identity Motivated by an Amazing Identity of Ramanujan ]
       {   An Identity Motivated by an Amazing Identity of Ramanujan}
\author{James Mc Laughlin}
\address{Mathematics Department\\
 25 University Avenue\\
West Chester University, West Chester, PA 19383}
\email{jmclaughl@wcupa.edu}


\date{\today}

\begin{abstract}
Ramanujan stated an  identity to the effect that if three sequences
$\{a_n\}$, $\{b_n\}$ and $\{c_n\}$ are defined by
$r_1(x)=:\sum_{n=0}^{\infty}a_nx^n$,
$r_2(x)=:\sum_{n=0}^{\infty}b_nx^n$ and
$r_3(x)=:\sum_{n=0}^{\infty}c_nx^n$ (here each $r_i(x)$ is a certain
rational function in $x$), then
\[
a_n^3+b_n^3-c_n^3=(-1)^n, \hspace{25pt} \forall \,n \geq 0.
\]

Motivated by this amazing identity, we state and prove a more
general identity involving eleven sequences, the new identity being
"more general" in the sense that equality holds not just for the
power 3 (as in Ramanujan's identity), but for each power $j$, $1\leq
j \leq 5$.
\end{abstract}

\maketitle

\section{Introduction}

In the ``lost notebook" \cite[page 341]{R88}, Ramanujan records the
following remarkable identity. If the sequences $\{a_n\}$, $\{b_n\}$
and $\{c_n\}$ are defined by
\begin{align}\label{rameq1}
\frac{1+53x+9x^2}{1-82x-82x^2+x^3}&=:\sum_{n=0}^{\infty}a_nx^n,\\
\frac{2-26x-12x^2}{1-82x-82x^2+x^3}&=:\sum_{n=0}^{\infty}b_nx^n,\notag\\
\frac{2+8x-10x^2}{1-82x-82x^2+x^3}&=:\sum_{n=0}^{\infty}c_nx^n,\notag
\end{align}
then
\begin{equation}\label{rameq2}
a_n^3+b_n^3-c_n^3=(-1)^n, \hspace{25pt} \forall \,n \geq 0.
\end{equation}

As Hirschhorn remarks in \cite{H95}, what is amazing about this
identity is not only that it is true, but that anyone could come up
with it in the first place. As well as giving a proof of the
identity, Hirschhorn  also gives a plausible explanation of how
Ramanujan might have discovered it. A second proof of the identity
was given by Hirschhorn in \cite{H96}, and a third proof was given
by Hirschhorn and Han in \cite{HH06}, where the authors also prove
that the sequences $\{a_n\}$, $\{b_n\}$ and $\{c_n\}$ may also be
derived from a certain matrix equation.

Motivated by this amazing identity of Ramanujan, and Hirschhorn
explanation of how Ramanujan might have found it, we present a more
general identity in the present paper, one where the three sequences
in \eqref{rameq2} are replaced by eleven sequences, and the
identity holds not just for a single exponent (3 in the case of
\eqref{rameq2}), but for all integer exponents $j$, $1 \leq j \leq
5$.

\section{A Ramanujan-type Identity  }
The identity referred to is described in the following theorem.

\begin{theorem}\label{t2}
Let the sequences of  integers $a_k$, $b_k$, $c_k$, $d_k$, $e_k$,
$f_k$,  $p_k$, $q_k$, $r_k$, $s_k$ and $t_k$   be defined by {\allowdisplaybreaks
\begin{align}\label{gfs2eq}
&\frac{x^2+164 x+3}{x^3-99 x^2+99 x-1}=:\sum_{k=0}^{\infty}a_k x^k,&&
\frac{-5 x^2+138 x+3}{x^3-99 x^2+99 x-1}=:\sum_{k=0}^{\infty}p_k x^k,&\notag \\
& \frac{-7 x^2+134 x+1}{x^3-99 x^2+99 x-1}
 =:\sum_{k=0}^{\infty}b_k x^k,&&\frac{3 x^2+244 x+1}{x^3-99 x^2+99 x-1}=:\sum_{k=0}^{\infty}q_k x^k,&\notag  \\
&\frac{-x^2+298 x-1}{x^3-99 x^2+99 x-1}=:\sum_{k=0}^{\infty}c_k x^k,&&\frac{x^2+254 x-7}{x^3-99 x^2+99 x-1}=:\sum_{k=0}^{\infty}r_k x^k,&\notag  \\
 &\frac{-5 x^2+228 x-7}{x^3-99 x^2+99 x-1}=:\sum_{k=0}^{\infty}d_k x^k,&&\frac{-7 x^2+148 x-5}{x^3-99 x^2+99 x-1}=:\sum_{k=0}^{\infty}s_k x^k,&\notag  \\
 &\frac{3 x^2+258 x-5}{x^3-99 x^2+99 x-1}=:\sum_{k=0}^{\infty}e_k x^k,&&\frac{3}{1-x}=:\sum_{k=0}^{\infty}t_k x^k,&\notag  \\
& \frac{-3 x^2+94 x-3}{x^3-99 x^2+99 x-1}=:\sum_{k=0}^{\infty}f_k x^k.&&&\notag
\end{align}
} Then for $1 \leq j \leq 5$, and each $k\geq 0$,
\begin{equation}\label{id2eq}
a_k^j+b_k^j+c_k^j+d_k^j+e_k^j+f_k^j
-p_k^j-q_k^j-r_k^j-s_k^j-t_k^j=1.
\end{equation}
\end{theorem}

We note that \eqref{id2eq} differs from Ramanujan's identity
\eqref{rameq2}, in that \eqref{id2eq} is true for each integer
exponent $j$, $1 \leq j \leq 5$, in contrast to \eqref{rameq2},
which is true only for the fixed exponent 3. For example, one can
check that
\begin{multline*}
\{ a_1,b_1,c_1,d_1,e_1,f_1,p_1,q_1,r_1,s_1,t_1\}\\
=\{-461,-233,-199,465,237,203,-435,-343,439,347,3 \}
\end{multline*}
and that
\begin{multline}\label{pteexeq}
(-461)^j+(-233)^j+(-199)^j+465^j+237^j+203^j\\
-(-435)^j-(-343)^j-439^j-347^j-3^j=1,
\end{multline}
for $1 \leq j \leq 5$. Like Ramanujan's sequences, the terms in our
sequences also grow arbitrarily large (except for $t_k$ which has
the constant value $3$ for all $k \geq 0$), while the left side of
\eqref{id2eq} maintains the constant value 1.

Many readers will no doubt have recognized that what has been
encoded in the various generating functions is a sequence of ideal
solutions of size 6 to what has become known as the \emph{Prouhet-Tarry-Escott} problem (Dickson \cite{D66} referred to it
as the problem of ``equal sums of like powers").
Before coming to the proof of Theorem \ref{t2}, we briefly discuss
this problem.

The \emph{Prouhet-Tarry-Escott} problem, which has a history going
back to Goldbach, asks for two distinct
 multisets of integers $A = \{a_1, . . . , a_m\}$
and $B = \{b_1, . . . , b_m\}$ such that
\begin{equation}\label{pteeq}
\sum_{i=1}^{m}a_i^e=\sum_{i=1}^{m}b_i^e, \text{ for } e = 1, 2,
\dots , k,
\end{equation}
for some integer $k < m$. We call $m$ the \emph{size} of the
solution and $k$ the \emph{degree}. If $k = m-1$, such a solution is
called \emph{ideal}. For example, it is easy to check that
\[
 1^j + 21^j + 36^j + 56 ^j =2^j + 18^j + 39^j + 55^j
\]
holds for $j=1,2$ and $3$. Thus $A=\{1, 21, 36, 56 \}$, $B = \{ 2,
18, 39, 55 \}$ provide an ideal solution of size 4.

We write
\begin{equation}\label{pteeq2}
\{a_1, . . . , a_m\}\stackrel{k}{=}\{b_1, . . . , b_m\}
\end{equation}
to denote a solution of size $m$ and degree $k$ to the
Prouhet-Tarry-Escott problem. As regards parametric solutions, an
early example was given by Euler (see \cite[page 705]{D66}), who showed that
\[
\{a,b,c,a+b+c\}\stackrel{2}{=}\{a+b,a+c,b+c,0\}.
\]

Parametric ideal solutions are known for $m=1,\dots , 8$ and
particular numerical solutions are known for $m=9, 10$ and 12. The
interested reader may find some of the early history of this interesting problem in Chapter XXIV of \cite{D66},  and some of the more recent developments
at \cite{BI94} and \cite{S07}.

The following parametric solution of size 6 is due to Chernick
\cite{C37}. For any integers $m_k$ and $n_k$, if
{\allowdisplaybreaks
\begin{align}\label{ab87eq}
&a_k'=-5 m_k^2+4 m_k n_k-3 n_k^2,&&p_k'=-5 m_k^2+6 m_k n_k+3 n_k^2,&\\
&b_k'=-3m_k^2+6 m_k n_k+5 n_k^2,&&q_k'=-3m_k^2-4 m_k n_k-5n_k^2,&\notag\\
&c_k'=- m_k^2-10 m_k n_k- n_k^2,&&r_k'=- m_k^2+10 m_k n_k- n_k^2,&\notag\\
&d_k'=5 m_k^2-4 m_k n_k+3 n_k^2,&&s_k'=5 m_k^2-6 m_k n_k-3 n_k^2,&\notag\\
&e_k'=3 m_k^2-6 m_k n_k-5 n_k^2,&&t_k'=3 m_k^2+4 m_k n_k+5 n_k^2,&\notag\\
&f_k'=m_k^2+10 m_k n_k+ n_k^2,&&u_k'= m_k^2-10 m_k n_k+ n_k^2.&\notag
\end{align}
} then
\begin{equation}\label{par8eq}
\{a',b',c',d',e',f'\}\stackrel{5}{=}\{p',q',r',s',t',u'\}.
\end{equation}

The observant reader will have noticed that the twelve terms
actually form 6 pairs, each of the three pairs on each side of
\eqref{par8eq} consisting of a term and its negative ($d_k'=-a_k'$
and so on), so that \eqref{par8eq} is trivially true for odd powers.
To make our generating functions and sequences at least superficially more interesting, we will modify these sequences  using the easily-proved fact that
if \[ \{a_1, . . . , a_m\}\stackrel{k}{=}\{b_1, . . . , b_m\},\]
then
\[ \{Ma_1+K, . . . , Ma_m+K\}\stackrel{k}{=}\{Mb_1+K, . . . ,
Mb_m+K\},\] for constants $M$ and $K$.

In the present case, we will determine particular sequences
$\{m_k\}_{k=0}^{\infty}$ and $\{n_k\}_{k=0}^{\infty}$, with the
sequences $a_k' \dots u_k'$ being defined by \eqref{ab87eq}. We then
set $a_k=a_k'+2u_k'$, $b_k=b_k'+2u_k'$ and so on. In particular,
$r_k=r_k'+2u_k'=-u_k'+2u_k'=u_k'$. We will further show that
$u_k'=r_k=1$, so that
\[
\{a_k,b_k,c_k,d_k,e_k,f_k\}\stackrel{5}{=}
\{p_k,q_k,1,s_k,t_k,u_k\}
\]
will hold automatically for each integer $k\geq 0$, which gives
\eqref{id2eq}, after a slight manipulation.

All that will remain will be to show that each of the generating
functions has the stated form. We now proceed to the proof.

\begin{proof}[Proof of Theorem \ref{t2}]
Set $h_0=0$, $h_1=1$, and for $k>1$, set
\begin{equation}\label{h2eq}
h_k=10 h_{k-1}- h_{k-2}.
\end{equation}
 Upon solving
the characteristic equation $x^2-10x+1=0$ and applying the stated
initial conditions, we find that
\begin{align*}
h_k&=
  -\frac{\left(5-2 \sqrt{6}\right)^k}{4 \sqrt{6}}+\frac{\left(5+2
   \sqrt{6}\right)^k}{4 \sqrt{6}},\\
h_k^2&=\frac{-2+\left(49-20 \sqrt{6}\right)^k+\left(49+20
   \sqrt{6}\right)^k}{96},\\
h_{k+1}h_k&=\frac{-10+\left(5-2
   \sqrt{6}\right)\left(49-20 \sqrt{6}\right)^k +\left(5+2 \sqrt{6}\right) \left(49+20
   \sqrt{6}\right)^k}{96}.
\end{align*}

In  \eqref{ab87eq}, we set $m_k=h_{k+1}$ and $n_k=h_k$, noting that
\eqref{h2eq} implies that
\begin{equation*}
h_{k+1}^2-10h_{k+1}h_k+h_k^2=h_{k}^2-10h_{k}h_{k-1}+h_{k-1}^2=\dots
=h_1^2-10h_1h_0+h_0^2=1,
\end{equation*}
so that $r_k=u_k'=1$. Thus all that remains is to show that,
with these choices for $m_k$ and $n_k$, that the various generating
functions have the stated forms. We do this for $\sum_{k=0}^{\infty}a_k x^k$ only, since
the proofs for the other generating functions are virtually
identical.

Define {\allowdisplaybreaks
\begin{align*}
H_1(x)&:= \sum_{k=0}^{\infty}h_k^2x^k =\sum_{k=0}^{\infty}\frac{-2+\left(49-20 \sqrt{6}\right)^k+\left(49+20
   \sqrt{6}\right)^k}{96}x^k\\
    &=\frac{1}{96}
    \left(\frac{-2}{1-x}
    +\frac{1}{1-  \left(49-20 \sqrt{6}\right)x}
    +\frac{1}{1-  \left(49+20
   \sqrt{6}\right)x}\right)\\
    &=\frac{-x (x+1)}{x^3-99 x^2+99 x-1},\\
H_2(x)&:= \sum_{k=0}^{\infty}h_{k+1} h_k x^k\\
    &=\frac{1}{96}
    \left(\frac{-10}{1-x}
    +\frac{5-2
   \sqrt{6}}{1-  \left(49-20 \sqrt{6}\right)x}
    +\frac{5+2 \sqrt{6}}{1-  \left(49+20
   \sqrt{6}\right)x}\right)\\
    &=\frac{-10 x}{x^3-99 x^2+99 x-1},\\
H_3(x)&:= \sum_{k=0}^{\infty}h_{k+1}^2x^k\\
&=\frac{H_1(x)}{x}=\frac{-x-1}{x^3-99 x^2+99 x-1}.
\end{align*}
}

These formulae for $H_1(x)$, $H_2(x)$ and $H_3(x)$ follow after
using the summation formula for an infinite geometric series a
number of times, and then using a little algebra to combine the
resulting rational expressions.

Next,
\[
a_k=a_k'+2u_k'=-5h_{k+1}^2+4h_{k+1}h_k-3 h_k^2+2,
\]
so that
\begin{align*}
\sum_{k=0}^{\infty}a_kx^k
&=\sum_{k=0}^{\infty}(-5h_{k+1}^2+4h_{k+1}h_k-3 h_k^2+2)x^k\\
&=
-5H_3(x)+4H_2(x)-3H_1(x)+\frac{2}{1-x}\\
&=\frac{x^2+164 x+3}{x^3-99 x^2+99 x-1},
\end{align*}
as claimed in Theorem \ref{t2}. The claimed formulae for the other
generating functions follow similarly, giving the result.
\end{proof}

 \allowdisplaybreaks{

}

\medskip

\noindent AMS Classification Numbers: 11A55

\end{document}